\documentclass[a4,12pt]{amsart}
\usepackage{amssymb,amstext,amsmath,amscd,amsthm,amsfonts,enumerate,graphicx,latexsym}
\usepackage[usenames]{color}
\usepackage[all]{xy}
\usepackage{ulem}
\newtheorem{thm}{Theorem}[section]
\newtheorem{cor}[thm]{Corollary}
\newtheorem{lem}[thm]{Lemma}
\newtheorem{prop}[thm]{Proposition}
\theoremstyle{definition}
\newtheorem{dfn}[thm]{Definition}
\newtheorem{thmdfn}[thm]{Theorem and Definition}
\newtheorem{rem}[thm]{Remark}

\newtheorem{ex}[thm]{Example}

\newtheorem*{claim*}{Claim}
\theoremstyle{remark}
\newtheorem*{ac}{Acknowlegments}

\numberwithin{equation}{thm}

\def\d{\partial}
\def\natu{^\natural}
\def\Ext{\mathrm{Ext}}
\def\Hom{\mathrm{Hom}}

\def\Z{\mathbb{Z}}

\begin{document}
\setlength{\baselineskip}{15pt}
\title{A lifting problem for DG modules}
\date{\today}
\author{Maiko Ono}
\address{Graduate School of Natural Science and Technology, Okayama University, Okayama, 700-8530,  Japan} 
\email{ptr56u9x@s.okayama-u.ac.jp}
\author{Yuji Yoshino}
\address{Graduate School of Natural Science and Technology, 
Okayama University, Okayama, 700-8530, Japan}
\email{yoshino@math.okayama-u.ac.jp}
\begin{abstract} 
Let $B = A\langle X | dX=t \rangle$  be an extended DG algebra by the adjunction of variable of positive even degree $n$,  and let  $N$  be a semi-free DG $B$-module that is assumed to be bounded below as a graded module. 
We prove in this paper that  $N$ is liftable to $A$ if   $\Ext_B^{n+1}(N,N)=0$. 
Furthermore such a lifting is unique up to DG isomorphisms if  $\Ext_B^{n}(N,N)=0$. 
\end{abstract}
\maketitle

\section{introduction}

Lifting problems of algebraic structures appear in various phases of algebra theory.  
In fact many authors have studied variants of liftings in their own fields such as modular representation theory, deformation theory and commutative ring theory etc. 
From the particular view point of ring theory,
the lifting problem and their weak variant called weak lifting problem  were systematically investigated firstly by M.Auslander, S.Ding and \O.Solberg \cite{ADS}.  
The second author of the present paper extended the lifting problems to chain complexes and developed a theory of weak liftings for complexes in \cite{Y}. 
On the other hand in the papers \cite{NS,NY},  the lifting or weak lifting problems were generalized into the corresponding problems for DG modules, however they only considered the cases of Koszul complexes that are DG algebra extensions by adding one variable of odd degree. 
In contrast, our main target in the present paper is the lifting problem for DG algebra extension obtained by adding a variable of  positive {\it even}  degree. 

Let $A$ be a commutative DG algebra over a commutative ring $R$, and $X$  be a variable of degree $n=|X|$. 
Then we can consider the extended DG algebra $A\langle X | dX=t \rangle$ that is obtained by the adjunction of variable $X$ with relation $dX =t$, 
 where $t$ is a cycle in $A$  of degree $n-1$. 
Note that if $n$ is odd,  then $A\langle X | dX=t \rangle = A \oplus XA$ as a right $A$-module, which is somewhat similar to a Koszul complex.
In contrast, if $n$ is even,  $A\langle X | dX=t \rangle = \bigoplus _{i \ge 0} X^{(i)}A$ is a free algebra over $A$ with divided variable $X$ that resembles a polynomial ring.  
In each case there is a natural DG algebra homomorphism $A \to A\langle X | dX=t \rangle$.  
See \S 2 below for more detail.  

In general,  let $A\to B$ be a DG algebra homomorphism. 
Then a DG $B$-module $N$ is said to be {\it liftable}  to $A$ if there is a DG $A$-module  $M$ with the property  $N \cong  B\otimes_A M$ as DG $B$-modules. 
In such a case $M$ is called a {\it lifting} of $N$. 
We are curious about the lifting problem for the particular case that $B = A\langle X | dX=t \rangle$. 
The both papers  \cite{NS, NY} treated the lifting problem in such cases but with the assumption that  $|X|$ is odd.
They actually showed that the vanishing of  $\Ext_B^{|X|+1}(N,N)$ implies the weak liftability of $N$. 

We consider the lifting problem for $A \to B = A\langle X | dX=t \rangle$ in the case that $|X|$ is positive and even. 
Surprisingly enough we are able to show in this paper that the vanishing of $\Ext_B^{|X|+1}(N,N)$ implies the liftablity (not a weak liftability) and moreover  the vanishing of $\Ext_B^{|X|}(N,N)$ implies the uniqueness of such a lifting.  
To say more precisely,  the following is our main theorem of this paper that answers the question raised in \cite[Remark 3.8]{NY}. 

\vspace{6pt} 
\noindent 
{\bf Theorem.}\ (Theorem \ref{main} and Theorem \ref{unique})
{\it 
Let  $A$ be a DG $R$-algebra, where $R$ is a commutative ring. 
Let  $B=A\langle X| dX=t\rangle$ denote a DG $R$-algebra obtained from $A$ by the adjunction of  variable $X$ of  
positive even degree. 
Further assume that $N$ is a semi-free DG $B$-module.
\begin{enumerate}
\item[$(1)$]  Under the assumption that $N$ is bounded below as a graded $R$-module, if $\Ext_B^{|X|+1}(N,N)=0$, then $N$ is liftable to $A$.
\item[$(2)$] If $N$ is liftable to $A$ and if $\Ext_B^{|X|}(N,N)=0$, then a lifting of $N$ is unique up to DG $A$-isomorphisms. 
\end{enumerate}
}
\vspace{6pt} 

In \S 2 we prepare the necessary definitions and notations for DG algebras and DG modules that will be used in this paper.
In \S 3 we introduce the notion of $j$-operator and give several useful properties of $j$-operators. 
\S 4 is a main body of the present paper, where we prove the  main theorem above. 

\begin{ac}
The authors are grateful to Dr. H. Minamoto for his many helpful suggestions.  
In fact the use of $j$-operators in our proof of the main theorem were  implicitly suggested  to the authors  by  him. 
The second author was supported by JSPS Grant-in-Aid for Scientific Research 26287008.
\end{ac}

\section{Preliminary on DG algebras and DG modules}

We summarize some definitions and notations that will be used in this paper.
Throughout the paper,  $R$ always denotes a commutative ring. 
Basically all modules considered in the paper are meant to be $R$-modules and all algebras are $R$-algebras.

Let $A=\bigoplus_{n\ge0}A_n$ be a non-negatively graded $R$-algebra equipped with a graded $R$-linear homomorphism $d^A:A \to A$ of degree $-1$. 
Then $A=\left( \bigoplus_{n\ge0}A_n, \ d^A \right)$ is called a {\it (commutative) differential graded $R$-algebra}, or a {\it DG $R$-algebra} for short, if it satisfies the following conditions:
\begin{enumerate}
\item For homogeneous elements $a$ and $b$ of $A$, $ab=(-1)^{|a||b|}ba$ where $|a|$ denotes the degree of $a$. 
Moreover if $|a|$ is odd, then $a^2=0$.
\item The graded $R$-algebra $A$ has a differential structure, by which we mean that $\left( d^A \right)^2=0$.
\item The differential $d^A$ satisfies the derivation property; $d^A(ab)=d^A(a)b+(-1)^{|a|}ad^A(b)$ for homogeneous elements $a$ and $b$ of $A$.
\end{enumerate}

Note that all DG algebras considered in this paper are non-negatively graded $R$-algebras. 
We often denote by  $A\natu$ the underlying graded $R$-algebra for a DG $R$-algebra $A$. 

Let $f:A \to B$ be a graded $R$-algebra homomorphism between DG $R$-algebras. 
By definition  $f$  is a {\it DG algebra homomorphism} if it is a chain map, i.e., $d^B f=f d^A$.

Let $A$ be a DG $R$-algebra and $M=\bigoplus_{n\in\Z}M_n$ be a graded left  $A$-module equipped with a graded $R$-linear map $\d^M:M \to M$ of degree $-1$.
Then $M=\left( \bigoplus_{n\in\Z}M_n,\d^M \right)$ is called a left {\it  differential graded $A$-module}, or a {\it DG $A$-module} for short,
if it satisfies the following conditions:
\begin{enumerate}
\item The graded module $M$ has a differential structure, i.e.,  $\left( \d^M \right)^2 = 0$.
\item The differential $\d^M$ satisfies the derivation property over $A$, i.e., $\d^M(am)=d^A(a)m+(-1)^{|a|}a\d^M(m)$ for $a\in A$ and $m\in M$.
\end{enumerate}
Note that every left differential graded $A$-module $M$  can be regarded as a right graded differential $A$-module by defining the right action as 
$ma = (-1)^{|a||m|} am$   for  $a \in A$  and  $m \in M$. 
Similarly to the case of DG algebras, $M\natu$ denotes the underlying graded $A\natu$-module for a DG $A$-module $M$.

To recall the definition of tensor products of DG modules, let $M$ and $N$ be DG $A$-modules. 
Then the graded tensor product  $M \natu \otimes_{A\natu} N\natu$ of graded modules  has the differential mapping defined by 
$$
\d^{M\otimes_AN}(m\otimes n) =\d^M(m)\otimes n +(-1)^{|m|} m\otimes \d^N(n) \quad \text{for} \ \  m \in M \ \  \text{and} \  \  n \in N .
$$
The tensor product  of DG $A$-modules is denoted by $M \otimes _A N$, by which we mean the DG $A$-module $\left( M \natu \otimes_{A\natu} N\natu,   \d^{M\otimes_AN}\right)$. 

If  $A \to B$ is a DG algebra homomorphism, and if $M$ is a DG $A$-module, then $B\otimes_A M$ is regarded as a DG $B$-module via action $b(b'\otimes m) =bb'\otimes m$ for $b, b' \in B$  and $m \in M$.

\begin{dfn}
Let $A\to B$ be a DG algebra homomorphism.
A DG $B$-module $N$ is called  {\it liftable} to $A$ if there exists a DG $A$-module $M$ such that $N$ is isomorphic to $B\otimes_A M$ as DG $B$-modules.
In this case, $M$ is called a {\it lifting} of $N$.
\end{dfn}

Let $A$ be a DG $R$-algebra and let $M$,  $N$ be DG $A$-modules.

A graded $R$-module homomorphism $f : M\to N$ of degree $r \ (r \in \Z)$ is, by definition, an $R$-linear mapping from $M$ to $N$ with  $f(M_n) \subseteq N_{n+r}$ for all $n \in \Z$. 
In such a case we denote  $|f| =r$. 
The set of all graded $R$-module homomorphisms of degree $r$ is denoted by $\Hom _R (M, N)_r$.  
Then $\Hom _R (M, N) = \bigoplus _{r \in \Z}  \Hom _R(M, N)_r$ is naturally a graded $R$-module.   
A graded $R$-module homomorphism $f \in \Hom _R (M, N)_r$ is called {\it $A$-linear} if it satisfies $f(am)=(-1)^{r|a|}af(m)$ for $a\in A$ and $m\in M$. 
We denote by $\Hom _A (M,N)_r$ the set of all $A$-linear homomorphisms of degree $r$. 
Then $\Hom_A(M,N)=\bigoplus_{r\in\Z}\Hom _A(M, N)_r$ has a structure of graded $A$-module, 
on which we can define the differential as follows: 
$$
\d^{\Hom_A(M,N)}(f) =\d^N f - (-1)^{|f|} f \d^M.
$$
In such a way we have defined the DG $A$-module $\Hom_A(M,N)$. 

By definition, a {\it DG $A$-homomorphism} $f : M \to N$  is  an $A$-linear homomorphism of degree $0$ that is a cycle as an element of  $\Hom_A(M,N)$.
A DG $A$-homomorphism $f:M\to N$ gives a {\it DG $A$-isomorphism}  if $f$ is invertible as a graded $A$-linear homomorphism. 
On the other hand a DG $A$-homomorphism $f : M \to N$ is called  a {\it quasi-isomorphism} if the homology mapping  $H(f) : H(M) \to H(N)$ is an isomorphism of graded $R$-modules.  

A DG $A$-module $F$ is said to be {\it semi-free} if $F\natu$ possesses a graded  $A\natu$-free basis $E$ which  decomposes as a disjoint union $E=\bigsqcup_{i\ge0} E_i$ of subsets indexed by natural numbers and satisfies  $\d^M(E_i)\subseteq \sum_{j < i} A E_{j}$ for $i \ge 0$.
A {\it semi-free resolution} of a DG $A$-module $M$ is a DG $A$-homomorphism $F\to M$ from a semi-free DG $A$-module $F$ to $M$,  which is a quasi-isomorphism.
It is known that any DG $A$-module has a semi-free resolution.
See \cite[Theorem 8.3.2 ]{AFH}.
Given a DG $A$-module $M$, and taking a semi-free resolution $F_M \to M$, one can define the $i$th extension module by  
 $$\Ext_A^i(M,N) := {H}_{-i}\left( \Hom_A(F_M,N)\right), $$  which is known to be independent of choice of a semi-free resolution of $M$ over $A$.
See \cite[Proposition 1.3.2.]{A}.

There is a well-known way of construction of a DG algebra that kills a cycle by adjunction of a variable.  
See \cite{A},\cite{G} and \cite{T} for details.
To make it more explicit, let $A$ be a DG $R$-algebra and take a homogeneous cycle $t$  in $A$.
We are able to construct an extended DG $R$-algebra $B$ of $A$ by the adjunction of a variable $X$ with $|X|=|t|+1$ which kills the cycle $t$ in the following way.
In both cases, we denote $B$ by $A\langle X|dX=t\rangle$.

\vspace{6pt}
\par\noindent 
(1) If $|X|$ is odd, then $B =A\oplus XA$ with algebra structure $X^2=0$ in which the differential is defined  by $d^B(a+Xb) =d^A(a)+tb-Xd^A(b)$.

\vspace{6pt}
\par\noindent 
(2) If $|X|$ is even, then $B=\bigoplus_{i\ge0}X^{(i)}A$ which is an algebra with divided powers of variable $X$.
Namely it has the multiplication structure $X^{(i)}X^{(j)}=\binom{i+j}{i}X^{(i+j)}$ for $i,j\in\Z_{\ge0}$ with $|X^{(i)}| = i|X|$. 
(Here we use the convention $X^{(0)}=1,X^{(1)}=X$.) 
Adding to the derivation property, the differential on $B$ is simply defined by the rule $d^B(X^{(i)})=tX^{(i-1)}$ for $i >0$, 
hence  it is given as follows for general elements: 
$$
d^B\left( \sum _{i=0}^n  X^{(i)}a_{i} \right) = \sum _{i=0} ^{n-1} X^{(i)} \left\{ d^A(a_i) + ta_{i+1}  \right\} + X^{(n)} d^A (a_n) .    
$$ 
\vspace{6pt}

In each case there is a natural DG $R$-algebra homomorphism  $\nolinebreak{A \to B = A\langle  X | dX =t \rangle}$. 
As we have mentioned in the introduction, we are interested in DG $B$-modules $N$ that are liftable to $A$, 
particularly in the case (2) above, that is  $|X|$  is even. 
In fact, S. Nasseh and Y. Yoshino have studied a liftable condition, or more generally a weak liftable condition,  for DG $B$-modules in the case where $|X|$ is odd. 
See \cite[Theorem 3.6]{NY} for more detail.

\section{the $j$-operator}

As in the last of the previous section, $A$ is a DG $R$-algebra in which we take a cycle $t$, that is, $d^At =0$. 
Our  specific assumption here is that $|t|$ is an odd non-negative integer. 
We denote by $B=A\langle X|dX=t\rangle$ the extended DG algebra of $A$ by the adjunction of variable $X$ that kills the  cycle $t$.
Since  $|X|$  is even,  note that 
\begin{equation}\label{decomposition of B} 
B =\bigoplus_{i\ge0}X^{(i)}A, 
\end{equation}
where the right hand side is a direct sum of right $A$-modules. 

Let $N$ be a DG $B$-module, and we assume the following conventional assumption:  
\begin{equation}\label{assumption} 
\text{There is a graded $A\natu$-module $M$ satisfying $N\natu = B\natu\otimes_{A\natu}M$.\hphantom{OOOOOOOO}} 
\end{equation}
\noindent 
Note that if $N$ is a semi-free DG $B$-module, then, since  $N\natu$  is a free $B\natu$-module,  it is always under such a circumstance.
By virtue of the decomposition (\ref{decomposition of B}),   we may write $N\natu$ as follows under the assumption (\ref{assumption}): 
\begin{equation}\label{decomposition of N} 
N\natu = \bigoplus_{i \ge0}X^{(i)}M.
\end{equation}
Note that there are equalities of $R$-modules 
$$
N_n = \bigoplus _{i \ge 0, \ k + i|X| = n} X^{(i)}M_k, 
$$
for all $n \in \Z$.

Now let  $r$ be an integer and let  $f \in \Hom _R(N, N)_r$. 
Recall that  $f$ is $R$-linear with  $f(N_n) \subseteq N_{n+r}$  for all $n \in \Z$.  
Given such an $f$, we consider the restriction of $f$ on $M$, i.e.  $f|_M \in \Hom _R (M, N)_r$. 
Along the decomposition (\ref{decomposition of N}),  one can decompose $f|_M$ into the following form: 
\begin{equation}\label{expansion}
f|_M= \sum_{i \ge0} X^{(i)}f_i,
\end{equation}
where each $f_i \in \Hom _R (M, M)_{r-i|X|}$.  
Actually, for $m \in M$, there is a unique decomposition $f (m) = \sum _{i} X^{(i)} m_i$  with $m_i \in M$ along (\ref{decomposition of N}). 
Then $f_i$  is defined by  $f_i(m) = m_i$.      
Note that  the decomposition (\ref{expansion}) is unique as long as we work under the fixed setting (\ref{decomposition of N}). 
We call the equality (\ref{expansion}) {\it the expansion of $f|_M$} and often call $f_0$ {\it the constant term of $f|_M$}.

Taking the expansion of $f|_M$ as in (\ref{expansion}), we consider the graded $R$-linear homomorphism   
\begin{equation*}
\varphi =\sum_{i\ge0}X^{(i)}f_{i+1},
\end{equation*}
which belongs to $\Hom _R (M, N) _{r-|X|}$. 
This $R$-linear mapping $\varphi$ can be extended to an $R$-linear mapping  $j(f)$  on $N$ by setting  
$j(f) (X^{(i)}m_i ) = X^{(i)} \varphi (m_i)$  for each  $i \geq 0$  and  $m_i \in M$. 
In such a way we obtain  $j(f) \in \Hom _R (N, N) _{r-|X|}$. 

Summing up the argument above we get the mapping  $j : \Hom _R (N, N) _r \to \Hom _R (N, N) _{r-|X|}$ for all $r \in \Z$, which we call {\it the $j$-operator} on $\Hom_R(N, N)$. 
For the later use we remark that the actual computation of $j(f)$  is carried out in the following way; 
\begin{equation}\label{jf}
j(f)(X^{(n)}m)=X^{(n)}j(f)(m)=X^{(n)}\sum_{i\ge0}X^{(i)}f_{i+1}(m)=\sum_{i\ge0}X^{(n+i)}\binom{n+i}{i}f_{i+1}(m)
\end{equation}
for $n\ge 0$ and $m\in M$.

\begin{thmdfn}\label{j-operator}
Under the assumption (\ref{assumption}) we can define a graded $R$-linear mapping  $j : \Hom _R (N, N) \to \Hom _R (N, N)$ of degree $-|X|$, 
which we call {\it the $j$-operator} on $\Hom _R(N, N)$. 
For any $f \in \Hom _R (N, N)_r$, taking the expansion (\ref{expansion}) of $f|_M$ along the decomposition (\ref{decomposition of N}), $j (f)$  maps  $X^{(n)} m$  to 
$\sum_{i\ge0}X^{(n+i)}\binom{n+i}{i}f_{i+1}(m)$  as in (\ref{jf}). 
\end{thmdfn}

\begin{rem}
The notion of $j$-operator was first introduced by J. Tate in the paper \cite{T} and extensively used by T.H. Gulliksen and G. Levin \cite{T}. 
\end{rem}

In the rest of the paper we always assume the condition  (\ref{assumption}) for a DG $B$-module $N$. 

\begin{dfn}
We denotes by $\mathcal{E}$ the set of all $B$-linear homomorphisms on $N$, i.e., $\mathcal{E}=\Hom_B(N,N)$. 
Note that  $\mathcal{E} \subseteq \Hom _R (N,N)$  and that a homogenous element $f \in \Hom _R(N, N)$ belongs to $\mathcal{E}$ if and only if $f (bn) = (-1)^{|b||f|} b f(n )$ for $b \in B$  and  $n \in N$.

We say that a graded $R$-linear mapping $\delta \in \Hom _R(N, N)$ is a {\it $B$-derivation} if it satisfies  $|\delta | =-1$ (i.e., $\delta \in \Hom _R (N, N)_{-1}$)  and 
$\delta(bn)=d^B(b)n+(-1)^{|b||\delta|}b\delta(n)$ for $b \in B$  and  $n \in N$. 
We denote by $\mathcal{D}$ the set of all $B$-derivations on $N$. 
\end{dfn}

\begin{rem}\label{remark on D} 
　\par
\begin{enumerate}
\item
Assume  that an $R$-linear mapping  $\delta : N \to N$ satisfies the derivation property $\delta(bn)=d^B(b)n+(-1)^{|b||\delta|}b\delta(n)$. 
 Since $|d^B(b)n| = |b| + |n| -1$  and $|b\delta(n)| = |b| +|n| + |\delta|$, if  $|\delta | \not= -1$  then $\delta$  is never a graded  mapping. 
\item
If  $\delta \in \mathcal{D}$  then the actual computation for $\delta$ is carried out by the following rule: 
\begin{equation*}
\delta\left(\sum_{i\ge 0}X^{(i)}m_i\right)=\sum_{i\ge 0}X^{(i)}\left\{ tm_{i+1}+\delta(m_i) \right\} . 
\end{equation*}
\item
If  $\delta, \delta ' \in \mathcal{D}$  then it is easy to see that  $\delta - \delta ' $  is in fact  $B$-linear, hence  $\delta - \delta '  \in \mathcal{E}$. 
\end{enumerate}
\end{rem}

Note that the both  $\mathcal{E}$  and  $\mathcal{D}$  are graded $R$-submodules of $\Hom _R(N, N)$. 

\begin{lem}\label{extension}
Under the assumption (\ref{assumption}), 
any $A$-linear homomorphism  $\alpha : M \to M$  is uniquely extended to $\tilde{\alpha} \in \mathcal{E}$ such that 
the constant term of the expansion of $\tilde{\alpha}  |_M$ equals $\alpha$. 
Similarly    
any $A$-derivation $\beta : M \to M$  is uniquely extended to $\tilde{\beta} \in \mathcal{D}$ such that 
the constant term of the expansion of $\tilde{\beta} |_M$ equals $\beta$. 

In both cases, we have  $j(\tilde{\alpha}) = 0$  and  $j(\tilde{\beta}) = 0$.   
\end{lem}

\begin{proof}
In each case the extension is obtained by making the tensor product with $B$ over $A$: 
$$
\tilde{\alpha}  = B \otimes _A \alpha, \quad \tilde{\beta}  = B \otimes _A \beta. 
$$ 
More precisely, any element $n \in N$ is written as $n = \sum_{i \ge 0}X^{(i)}m_i$ for $m_i \in M$ along (\ref{decomposition of N}), and 
taking into account the linearity of $\tilde{\alpha}$  and the derivation property of $\tilde{\beta}$, we can define them by the following equalities:  
$$
\tilde{\alpha} (n) = \sum_{i\ge 0}X^{(i)}\alpha (m_i), \quad \tilde{\beta} (n) = \sum_{i\ge 0}X^{(i)}\left\{ tm_{i+1}+\beta(m_i) \right\}. 
$$ 
Their uniqueness follows from the next lemma. 
\end{proof}

\begin{lem}\label{AB-linear}
Assume that $f, g \in \mathcal{E}$ and  $\delta, \delta' \in \mathcal{D}$.
Then the following assertions hold. 
\begin{enumerate}
\item[$(1)$]  $f=g$ if and only if $f|_M=g|_M$.
\item[$(2)$]  $\delta=\delta'$ if and only if $\delta|_M=\delta'|_M$.
\end{enumerate}
\end{lem}

\begin{proof}
$(1)$
Assume $f|_M=g|_M$.
For each $n\in N $ we decompose it into the form $n=\sum_{i \ge 0}X^{(i)}m_i$ along the decomposition (\ref{decomposition of N}). 
Since $f$ and $g$ are  $B$-linear, we have   
\begin{equation*}
f\left(\sum_{i \ge 0}X^{(i)}m_i\right)=\sum_{i \ge 0}X^{(i)}f(m_i)=\sum_{i\ge 0}X^{(i)}g(m_i)=g\left(\sum_{i \ge 0}X^{(i)}m_i\right), 
\end{equation*}
and hence $f=g$.

$(2)$ 
Assume $\delta|_M=\delta'|_M$.
Noting from Remark \ref{remark on D}(3)  that $\delta-\delta'$ is a $B$-linear homomorphism,
we see that $\delta-\delta'=0$ by virtue of $(1)$. 
\end{proof}

\begin{lem}\label{ED-prop} 
The following assertions hold.
\begin{enumerate}
\item[$(1)$]  If  $f \in \mathcal{E}$, then $j(f) \in \mathcal{E}$.
\item[$(2)$]   If  $\delta \in \mathcal{D}$ then $j(\delta) \in \mathcal{E}$.
\item[$(3)$]  Let  $\delta \in \mathcal{D}$. Then the constant term $\delta _0$ of the expansion of $\delta |_M$ is an $A$-derivation on $M$. 
\end{enumerate} 
As a consequence, the $j$-operator defines a mapping  $\mathcal{E} \cup \mathcal{D} \to \mathcal{E}$.  
\end{lem}

\begin{proof}
$(1)$ 
Write $f|_M=\sum_{i \ge 0}X^{(i)}f_i$ as in (\ref{expansion}).
Since $f$ is  $B$-linear and noting that  $|f| \equiv |f_i|(\rm{mod} \ 2)$, we see that  
$$f(am)=(-1)^{|a||f|}af(m)=(-1)^{|a||f|}a\sum_{i \ge 0}X^{(i)}f_i(m)=\sum_{i\ge 0}X^{(i)}(-1)^{|a||f_i|}af_i(m)$$
for $a \in A$ and $m\in M$. 
Thus by the uniqueness of expansion it is easy to see that 
$$f_i(am)=(-1)^{|a||f_i|}af_i(m).$$
Namely each  $f_i$ is $A$-linear, and therefore $j(f)|_M=\sum_{i \ge 0}X^{(i)}f_{i+1}$ is $A$-linear as well. 
Meanwhile,  it follows from the definition of  $j(f)$ or (\ref{jf}) that $j(f)$ commutes with the action of $X$ on $N$. 
Thus $j(f)$ is $B$-linear, and $j(f)$ belongs to $\mathcal{E}$.

$(2), (3); $ 
Write $\delta|_M=\sum_{i\ge0}X^{(i)}\delta_i$.
Since $\delta$ is a $B$-derivation,  we have equalities;  
\begin{equation*}
\delta(am)=d^B(a)m+(-1)^{|a||\delta|}a\delta(m) 
 =\left\{ d^A(a)m+(-1)^{|a||\delta|}a\delta _0(m) \right\}+\sum_{i \ge 1}X^{(i)}(-1)^{|a||\delta|}a\delta _i(m), 
\end{equation*}
for $a \in A$ and $m\in M$. 
On the other hand, $\delta(am)=\sum_{i \ge 0}X^{(i)}\delta_i(am).$
Comparing these equalities and noting that  $|\delta| \equiv |\delta _i|(\rm{mod} \ 2)$  for all $i \geq 0$, we eventually have  
$$
\delta_0(am)=d^A(a)m+(-1)^{|a||\delta _0|}a\delta_0(m)
\ \ \text{ and }\ \ 
\delta_i(am)=(-1)^{|a||\delta _i|}a\delta_i(m) \ \  \text{for} \ \ i >0, 
$$
which implies the required results in (2) and (3). 
\end{proof}

\begin{prop}\label{j-prop}
The following equalities hold for $f, g \in \mathcal {E}$ and $\delta, \delta' \in \mathcal{D}$.
\begin{enumerate}
\item[$(1)$]  $j(fg)=j(f)g+fj(g)$.
\item[$(2)$]  $j(f\delta)|_M=j(f)\delta|_M+fj(\delta)|_M$.
\item[$(3)$]  $j(\delta f)|_M=j(\delta)f|_M+\delta j(f)|_M$.
\item[$(4)$]  $j(\delta\delta')|_M=j(\delta)\delta'|_M+\delta j(\delta')|_M$.
\end{enumerate}
\end{prop}

Before proving Proposition \ref{j-prop}, 
we should remark that graded $R$-module homomorphisms $f\delta$, $\delta f$ and $\delta\delta'$ do not necessarily  belong to $\mathcal{E}$ or $\mathcal{D}$, and neither do $j(f\delta),j(\delta f)$ and $j(\delta\delta')$.
The equalities in $(2)$-$(4)$  hold only when they are restricted on $M$.

\begin{proof}
$(1)$ Note from Lemma \ref{ED-prop} that  $j(fg)$, $j(f)g$ and $fj(g)$ are elements of  $\mathcal{E}$.
By this reason, we have only to show that  $j(fg)|_M=j(f)g|_M+fj(g)|_M$ by Lemma \ref{AB-linear}.
Taking the expansions as  $f|_M = \sum _{i \ge 0} X^{(i)} f_i$ and 
$g|_M = \sum _{i \ge 0} X^{(i)} g_i$, we have the equalities: 
\begin{equation*}
fg|_M  = f(\sum_{i\ge 0}X^{(i)}g_i)=\sum_{i\ge 0}X^{(i)}fg_i=\sum_{i\ge 0}X^{(i)}\sum_{j\ge 0}X^{(j)}f_jg_i=\sum_{n\ge 0}X^{(n)}\sum_{i=0}^{n}\binom{n}{i}f_ig_{n-i}. 
\end{equation*}
Hence it follows from the definition of $j$-operator that 
\begin{equation*}
j(fg)|_M  = \sum_{n\ge 0}X^{(n)}\sum_{i=0}^{n+1}\binom{n+1}{i}f_ig_{n-i+1}.
\end{equation*}
On the other hand, we have equalities;  
\begin{equation*}
\begin{split}
j(f)g|_M+fj(g)|_M &=\sum_{k\ge 0}j(f)(X^{(k)}g_k)+\sum_{k\ge 0}f(X^{(k)}g_{k+1})\\
&=\sum_{k\ge0}\sum_{i\ge0} X^{(i+k)}\binom{i+k}{i}f_{i+1}g_k+\sum_{k\ge0}\sum_{i\ge0} X^{(i+k)}\binom{i+k}{i} f_i g_{k+1}\\
&=\sum_{n\ge0} X^{(n)} \left\{\sum_{i=1}^{n+1}  \binom{n}{i-1} f_ig_{n-i+1} + \sum_{i=0}^{n}\binom{n}{i} f_ig_{n-i+1}\right\}.
\end{split}
\end{equation*}
Since $\binom{n+1}{i}=\binom{n}{i-1}+\binom{n}{i}$ for $0 <  i \le n$, we deduce that $j(fg)|_M=j(f)g|_M+fj(g)|_M$. 

$(2)$ 
Recall from the previous lemma that the constant term $\delta_0$ in the expansion $\delta|_M=\sum_{i\ge0}X^{(i)}\delta_i$ is an $A$-derivation on $M$. 
Set  $\tilde{\delta_0}$ as the extended $B$-derivation of $\delta _0$ on $N$ defined by means of Lemma \ref{extension}. 
Then as we noted in Remark \ref{remark on D} (3),  $\delta-\tilde{\delta_0}$ is $B$-linear of degree $|\delta|=-1$. 
Moreover we see that $j(\delta-\tilde{\delta_0})=j(\delta)$, since $j(\tilde{\delta_0})=0$.
Thus it follows from the equality $(1)$ of the present lemma that 
$j((\delta-\tilde{\delta_0})f)=j(\delta-\tilde{\delta_0})f+(\delta-\tilde{\delta_0})j(f)=j(\delta)f+\delta j(f)-\tilde{\delta_0}j(f)$.
On the other hand, $j((\delta-\tilde{\delta_0})f)=j(\delta f)-j(\tilde{\delta_0}f)$.
Therefore we have that 
$$j(\delta f)-j(\tilde{\delta_0}f)=j(\delta)f+\delta j(f)-\tilde{\delta_0}j(f).$$
Hence it is enough to prove the equality in the case where $\delta=\tilde{\delta_0}$, that is, 
\begin{equation}\label{proving equality}
j(\tilde{\delta_0}f)|_M=\tilde{\delta_0}j(f)|_M.
\end{equation}
To prove this,  take the expansion as  $f|_M = \sum _{i \ge 0} X^{(i)} f_i$, and we get 
$$\tilde{\delta_0}f|_M  = \tilde{\delta_0}\left(\sum_{i\ge 0}X^{(i)}f_{i}\right) = \sum_{i\ge 0}X^{(i)}(tf_{i+1}+\delta_0f_i).$$
Then it follows that 
$$j(\tilde{\delta_0}f)|_M = \sum_{i\ge 0}X^{(i)}(tf_{i+2}+\delta_0f_{i+1}),$$
while  
\begin{equation*}
\tilde{\delta_0}j(f)|_M = \tilde{\delta_0}\left(\sum_{i\ge 0}X^{(i)}f_{i+1}\right) = \sum_{i\ge 0}X^{(i)}(tf_{i+2}+\delta_0f_{i+1}).
\end{equation*}
This proves (\ref{proving equality}). 

$(3)$  
Similarly to $(2)$, it is sufficient to prove the equality  
$j(f\tilde{\delta_0})|_M=j(f)\tilde{\delta_0}|_M$. 
If  $f|_M = \sum _{i \ge 0} X^{(i)} f_i$ is the expansion, then we  have 
$f\tilde{\delta_0}|_M=f\delta_0=\sum_{i\ge0}X^{(i)}f_{i}\delta_0$. 
Hence it follows from the definition of $j$-operator that  
$$j(f\tilde{\delta_0})|_M=\sum_{i\ge0}X^{(i)}f_{i+1}\delta_0=j(f)\delta_0=j(f)\tilde{\delta_0}|_M, $$
as desired. 

$(4)$ 
Let $\delta_0$ be the constant term of the expansion $\delta|_M=\sum_{i\ge0}X^{(i)}\delta_i$. 
As in the proof of (2) we take the extension $\tilde{\delta_0}$ of $\delta _0$, and hence it holds that  $\delta-\tilde{\delta_0}$ is $B$-linear, and that $j(\delta-\tilde{\delta_0})=j(\delta)$.
Now applying the equality proved in (2), we have that 
\begin{equation}\label{4-1}
j((\delta-\tilde{\delta_0})\delta' )|_M
 =j(\delta-\tilde{\delta_0})\delta' |_M+(\delta-\tilde{\delta_0})j(\delta' )|_M\\
 =j(\delta)\delta'|_M +\delta j(\delta')|_M-\tilde{\delta_0}j(\delta')|_M.
\end{equation}
In contrast, we have
\begin{equation}\label{4-2}
\begin{split}
j((\delta-\tilde{\delta_0})\delta')=j(\delta\delta')-j(\tilde{\delta_0}\delta').
\end{split}
\end{equation}
Combining these equalities, we obtain the equality:   
$$j(\delta\delta') |_M
=j(\delta)\delta'|_M+\delta j(\delta')|_M+j(\tilde{\delta_0}\delta')|_M-\tilde{\delta_0}j(\delta')|_M.
$$
Thus it is enough to prove the following equality: 
\begin{equation}\label{4-3}
j(\tilde{\delta_0}\delta') |_M=\tilde{\delta_0}j(\delta')|_M.
\end{equation}
To prove (\ref{4-3}) let  $\delta' |_M = \sum_{i\ge 0}X^{(i)}\delta'_i$ be the expansion, and we have that 
$$\tilde{\delta_0}\delta' |_M=\tilde{\delta_0}\sum_{i\ge 0}X^{(i)}\delta'_i=\sum_{i\ge 0}X^{(i)}(t\delta'_{i+1}+\delta_0\delta'_i), 
$$
therefore it follows 
$$j(\tilde{\delta_0}\delta')|_M=\sum_{i\ge 0}X^{(i)}(t\delta'_{i+2}+\delta_0\delta'_{i+1}).$$
On the other hand, we have
$$\tilde{\delta_0}j(\delta')|_M=\tilde{\delta_0}\sum_{i\ge 0}X^{(i)}\delta'_{i+1}=\sum_{i\ge 0}X^{(i)}(t\delta'_{i+2}+\delta_0\delta'_{i+1}).$$
It proves the equality  (\ref{4-3}). 
\end{proof}

\begin{cor}\label{j-cor}
Let $f \in \mathcal{E}$, and assume that $f$ is invertible in $\mathcal{E}$. 
Then we have equalities$:$
\begin{enumerate}
\item[$(1)$]  $j(f)f^{-1}+fj(f^{-1})=0.$
\item[$(2)$]  $j(f\delta f^{-1})=j(f)\delta f^{-1} + f j(\delta)f^{-1} + f \delta j(f^{-1})$.
\end{enumerate}
\end{cor} 

\begin{proof}
$(1)$
It follows from Proposition \ref{j-prop}(1) that 
$j(ff^{-1})=j(f)f^{-1}+fj(f^{-1})$. 
On the other hand,  it holds $j(ff^{-1})=j(\mathrm{id}_N)=0$, hence the equality (1) follows.

$(2)$
First of all we note that both $j(f)\delta f^{-1}+f\delta j(f^{-1})$ and $fj(\delta)f^{-1}$ are  $B$-linear. 
To verify this fact we remark that the following equalities hold: 
\begin{equation*}
\begin{split}
&(j(f)\delta f^{-1}+f\delta j(f^{-1}))(X^{(n)}m) 
= j(f)\delta (X^{(n)}f^{-1}(m))+f\delta(X^{(n)} j(f^{-1})(m)) \\
&= j(f)\{ tX^{(n-1)} f^{-1}(m) + X^{(n)} \delta f^{-1}(m) \} + f \{tX^{(n-1)}j(f^{-1})(m) + X^{(n)} \delta j(f^{-1})(m)\} \\
&= tX^{(n-1)} j(f)f^{-1}(m) + X^{(n)} j(f)\delta f^{-1}(m) + tX^{(n-1)} f j(f^{-1})(m) + X^{(n)} f \delta j(f^{-1})(m)\\
&= X^{(n)}(j(f)\delta f^{-1}+f\delta j(f^{-1}))(m), 
\end{split}
\end{equation*}
where the last equality holds because of (1).
On the other hand, since $f\delta f^{-1}$ belongs to $\mathcal{D}$,  $j(f\delta f^{-1})$ is $B$-linear as well. 
Therefore it is enough to prove the equality: 
$j(f\delta f^{-1})|_M=(j(f)\delta f^{-1}+fj(\delta)f^{-1}+f\delta j(f^{-1}))|_M$  by Lemma \ref{AB-linear}.  
From Proposition \ref{j-prop}(2), we get 
\begin{equation}\label{2}
j(f^{-1}(f\delta f^{-1}))|_M=j(f^{-1})(f\delta f^{-1})|_M+f^{-1}j(f\delta f^{-1})|_M.
\end{equation}
Meanwhile, Proposition \ref{j-prop}(3) implies that  
\begin{equation}\label{3}
j(f^{-1}(f\delta f^{-1}))|_M=j(\delta f^{-1})|_M=j(\delta)f^{-1}|_M+\delta j(f^{-1})|_M.
\end{equation}
Summarizing (\ref{2}) and (\ref{3}), we see that 
\begin{equation*}
\begin{split}
j(f\delta f^{-1})|_M 
&= -f j(f^{-1})(f\delta f^{-1})|_M + f j(\delta)f^{-1}|_M+f \delta j(f^{-1})|_M\\
&= j(f)\delta f^{-1}|_M + f j(\delta)f^{-1}|_M+f \delta j(f^{-1})|_M.
\end{split}
\end{equation*}
where the last equality holds by virtue of (1) in the present corollary. 
This completes the proof. 
\end{proof}

\section{main results}
Now we are able to prove the main theorems of this paper. 
See Theorem \ref{main} and Theorem \ref{unique} below. 

In the rest of the paper, $A$ always denotes a DG $R$-algebra and $B=A\langle X|dX=t\rangle$ is an extended DG algebra by the adjunction of variable $X$ that kills the  cycle $t \in A$, where  $|X|$  is a positive even integer.
Let  $N$  be a DG $B$-module and we always assume here that  $N$ is semi-free. 
We are interested in the conditions that sufficiently imply the liftability of $N$ to $A$. 
Since  $N^{\natu}$ is free as a $B\natu$-module, the condition (\ref{assumption}) is satisfied, that is,  
there is a graded $A\natu$-module $M$ such that $N\natu\cong B\natu\otimes_{A\natu}M$ as graded $B\natu$-modules.

The differential mapping $\d^N$ on $N$  belongs to $\mathcal{D}$ which, we recall, is the set of all $B$-derivations on $N$. 
It thus follows from Lemma \ref{ED-prop} that $j(\d^N)$ is $B$-linear, equivalently $j (\d ^N) \in  \mathcal{E}$. 
This specific element of  $\mathcal{E}$  will be a key object when we consider the lifting property of $N$ in the following argument. 
This is the reason why we make the following definition of  $\Delta _N$ as  

\begin{equation}\label{def of delta}
\Delta_N:=j(\d^N).
\end{equation}

Recall again from Lemma \ref{ED-prop} that $\Delta_N$ is a $B$-linear homomorphism on $N$ such that 
$| \Delta _N | = -|X|-1$ is an odd integer.

\begin{rem}
The exactly same definition was made by S. Nasseh and Y. Yoshino in the case where $|X|$ is odd. 
See \cite[Convention $2.5$]{NY}.
\end{rem}

As we see in the next lemma,  $\Delta_N$ defines an element of $\Ext^{|X|+1}_B(N,N)$, which will turn out to be an obstruction for the lifting of $N$ to $A$.

\begin{lem}\label{cyc}
It holds that  $\Delta _N \d^N = - \d ^N \Delta _N$. 
Hence  $\Delta_N$ is a cycle of degree $-|X|-1$ in $\mathcal{E} = \Hom_B(N,N)$.
\end{lem}

\begin{proof}
Noting that $\left( \d^N\right) ^2 =0$, we have from Proposition \ref{j-prop} that 
$0=j(\d^N\d^N)|_M=j(\d^N)\d^N|_M+\d^Nj(\d^N)|_M.$
On the other hand it is easily seen that $j(\d^N)\d^N+\d^Nj(\d^N)$ is $B$-linear. 
Hence it follows from Lemma \ref{AB-linear} that  $j(\d^N)\d^N+\d^Nj(\d^N)=0$. 
\end{proof}

In the proof of our main theorems, we shall need some argument on automorphisms on the DG $B$-module $N$. 
The following lemma is a preliminary for that purpose. 

\begin{lem}\label{conv-isom}
Let $\varphi:N\to N$ be a graded $B$-linear endomorphism of degree $0$. 
As before we assume that the expansion is given as $\varphi|_M=\sum_{i \ge 0}X^{(i)}\varphi_i$.
If  $\varphi$ is a $B$-linear automorphism on $N$, then the constant term $\varphi_0$ is an $A$-linear automorphism on $M$. 
\end{lem}

\begin{proof}
Take a graded $B$-linear endomorphism $\psi$ such that $\varphi\psi=\mathrm{id}_N = \psi\varphi$.
Writing $\psi|_M=\sum_{n \ge 0}X^{(n)}\psi_n$ as well, we see that  $\varphi\psi|_M=\mathrm{id}_M$ implies that 
the constant term $\varphi_0\psi_0$ of $\varphi\psi|_M$ is equal to $\mathrm{id}_M$.
Similarly $\psi_0\varphi_0=\mathrm{id}_M$.
Therefore $\varphi_0$ is an $A$-linear automorphism on $M$.
\end{proof}

A DG module $L$, or more generally a graded module $L = \bigoplus _{i \in \Z} L_i$,  is said to be {\it bounded below} if  $L_{-i}=0$ for all suffieceintly large  integers $i$.  
A graded endomorphism $f$ on a graded module $L$ is said to be {\it locally nilpotent} if,  
for any $x\in L$,  there is an integer $n_x \ge 0$ such that $f^{n_x}(x)=0$, where $f^{n_x}$ denotes the $n_x$ times iterated composition of $f$. 

The converse of Lemma \ref{conv-isom}  holds in several cases. 
The following is one of such cases.  

\begin{lem}\label{isom}
Adding to the assumption (\ref{assumption}) we further assume that $N$ is bounded below. 
Let $\varphi : N  \to N$ be a graded $B$-linear endomorphism of degree $0$ with expansion $\varphi|_M=\sum_{i \ge0}X^{(i)}\varphi_i$.
Assume that the constant term $\varphi_0$ is an $A$-linear automorphism on $M$.
Then $\varphi$ is a $B$-linear automorphism on $N$. 
\end{lem}

\begin{proof}
Note that $\varphi$ is an automorphism  if and only if so is $(B\otimes_A \varphi_0^{-1})\varphi$.
Hence we may assume $\varphi_0=\mathrm{id}_M$.
Setting  $f = \varphi-\mathrm{id}_N$, we are going to prove that $f$ is locally nilpotent.
For this we note that 
$f(M) \subseteq \bigoplus _{i \ge 1} X^{(i)}M$. 
Then, since  $f$  is $B$-linear,  we can show by induction on $n > 0$ that 
$$
f^n(M)  \subseteq  \bigoplus _{i \ge n} X^{(i)}M.  
$$
Since  $f$ has degree $0$ as well as $\varphi$,  
the graded piece  $M_r$  of $M$ of degree $r$ is mapped by $f^n$ into 
$$
\left( \bigoplus _{i \ge n} X^{(i)}M \right)_r = \bigoplus _{i \ge n} X^{(i)} M _{r - i|X|}.    
$$ 
Since  $M \subseteq N$ is a graded $A$-submodule, we remark that $M$ is also bounded below. 
For a given integer $r$, we can take an integer  $n$ that is large enough so that  $M_{r - i|X|} =0$  for all $i \ge n$, 
since $|X| > 0$. 
We thus have from the above that $f^n(M_r) = 0$.  
This shows that $f$ is locally nilpotent as desired.

Then  $\sum _{i=0}^{\infty} (-1)^i f^i =\mathrm{id}_N-f+f^2-f^3+\cdots+(-1)^nf^n+\cdots$ is a well-defined $B$-linear homomorphism on $N$, and in fact it is an inverse of $\varphi=\mathrm{id}_N+f$.
\end{proof}

The following is a key to prove one of the main theorems.

\begin{prop}\label{keylem}
Let $f$ be a graded $B$-linear endomorphism of degree $-|X|$ on $N$ 
and $g_0$ be a graded $A$-linear homomorphism of degree $0$ on $M$.
Then there is a graded $B$-linear endomorphism $g$ of degree $0$ on $N$  satisfying   that 
$$
j(g)=gf \text{ and  } g_0  \text{ is the constant term of  } g.
$$
\end{prop}

\begin{proof}
Take the expansion of $f$ as 
$
f|_M=\sum_{n\ge0}X^{(n)}f_n. 
$
Note here that each $f_n$ is a graded $A$-linear endomorphism on $M$ of degree $-|X|(n+1)$ for $n\ge0$. 
Setting $g|_M=\sum_{n\ge0}X^{(n)}g_n$, we want to determine each $g_n$ so that $g$ satisfies the desired conditions.

Recall that $j(g)|_M=\sum_{n\ge0}X^{(n)}g_{n+1}$  and that $gf|_M=\sum_{n\ge0}X^{(n)}\sum_{0\le i \le n}\binom{n}{i} g_i f_{n-i}$.
Comparing these equalities, we obtain the following equations for $g_n \ (n \ge 0)$ to satisfy the required conditions; 
$$
g_{n+1}=\sum_{i=0}^{n}\binom{n}{i}g_i f_{n-i} \text{ for all } n \ge 0.
$$
Starting from $g_0$ and using these equalities,  we can determine the graded $A$-linear homomorphism $g_n$ by the induction on $n \ge 0$. 
Thus define $g$ as a  linear extension of $g|_M$  to $N$, that is,  $g=B\otimes_{A}g|_M$. 
This is a $B$-linear endomorphism on $N$ of degree $-|X|$, and satisfies all the desired conditions.
\end{proof}

\begin{lem}\label{eq-lem}
Suppose that  $\Delta_{N}=0$ as an element of  $\mathcal{E}$. 
Then the graded $A$-module $M$  has structure of DG $A$-module and 
$N = B \otimes _A M$ holds as an equality of DG $B$-modules. 
\end{lem}

\begin{proof}
In the expansion  $\d^N |_M = \bigoplus _{i \ge 0} X^{(i)} \alpha _i$,   that $\Delta_{N}=0$ implies that 
$\alpha _i =0$ for $i >0$. 
Therefore  $\d^N |_M = \alpha _0$ is an $A$-derivation on $M$ and $(M, \alpha _0)$  defines a DG $A$-module. 
Moreover we have  $\d^N = B \otimes _A \alpha _0$ that equals $\tilde{\alpha}$ in the notation of Lemma \ref{extension}. 
Thus $N = B\otimes_A M$ as DG $B$-modules. 
\end{proof}

Now we are ready to prove the main theorem.
Note from Lemma \ref{cyc} that  $\Delta_N$ defines a cohomology class in $\Ext^{|X|+1}_B(N,N)$,  which we denote  by $[\Delta_N]$.
As we show in the following theorem the class $[\Delta _N]$  gives a precise obstruction for $N$ to to be liftable.

\begin{thm}\label{main}
As before let $N$ be a semi-free DG $B$-module, and assume that $N$ is bounded below.
Then $[\Delta_N]=0$ as an element of $\Ext^{|X|+1}_B(N,N)$ if and only if $N$ is liftable to $A$.
\end{thm}

\begin{proof} 
First of all we recall that $N$ is liftable if and only if there is an $A$-derivation $\d^M$ on $M$ of degree $-1$ that makes $(M, \d^M)$ a DG $A$-module and 
there is a DG $B$-isomorphism  $\varphi : N \to B \otimes _A M$. 
In such a case $\varphi$ is a graded $B$-linear isomorphism of degree $0$ that commutes with differentials, i.e., 
$(B \otimes _A \d ^M) \varphi =  \varphi \d ^N$ or equivalently 
\begin{equation}\label{4.7.1}
B \otimes _A \d ^M =  \varphi \d ^N \varphi ^{-1}.
\end{equation}

Now assume that $N$ is liftable. 
Then there is such a DG $B$-isomorphism $\varphi$. 
Taking the $j$-operator for (\ref{4.7.1}) and using Corollay \ref{j-cor} (2), we have that 
$$
0 = j (B \otimes _A \d ^M) =  j(\varphi)  \d ^N \varphi ^{-1} 
+\varphi j(\d ^N)  \varphi ^{-1} + \varphi \d ^N j( \varphi ^{-1}) . 
$$
It thus follows that  $$j(\d ^N) = - \varphi ^{-1} j (\varphi) \d ^N - \d^N j(\varphi ^{-1}) \varphi. $$ 
Here we note form Corollary \ref{j-cor} (1) that  $\varphi ^{-1} j (\varphi) = - j(\varphi ^{-1}) \varphi$.
Therefore if we set  $f = \varphi ^{-1} j(\varphi )$,  then we see that  $|f|=-|X|$ is even and  $\Delta _N = j(\d ^N) = \d ^N f-f \d ^N $. 
The last equality shows  $[\Delta _N]=0$  in $\Ext _B^{|X|+1} (N, N)$.

Contrarily assume that $[\Delta_N]=0$.
Then there is a graded $B$-linear endomomorphism $\gamma$ on $N$ of degree $-|X|$,  which satisfies the equality 
\begin{equation}\label{delta}
\Delta_N = \d^N \gamma - \gamma \d^N. 
\end{equation}
We note that $|\Delta_N|$ is odd and $|\gamma|$ is even. 
It follows from Proposition \ref{keylem} that there is a $B$-linear endomorphism $\varphi$ on $N$ of degree $0$ such that $\varphi_0=\mathrm{id}_M$ and
\begin{equation}\label{key}
j(\varphi) = \varphi\gamma.
\end{equation}
We should note from Lemma \ref{isom} that such $\varphi$ is a $B$-linear automorphism on $N$.
Define an alternative differential $\d'^N$ on $N$ by 
$$\d'^{N}=\varphi\d^N\varphi^{-1}.$$
Then it follows that $\varphi:(N,\d^N)\to(N,\d'^{N})$ is a DG $B$-isomorphism.

Since the equality $j(\varphi^{-1})\varphi+\varphi^{-1}j(\varphi)=0$
holds by Corollary \ref{j-cor} $(1)$, 
we see from  (\ref{key}) that  
\begin{equation}\label{j2}
j(\varphi^{-1})=-\gamma \varphi^{-1}.
\end{equation}
Thus we conclude that 
$$j(\d'^N)=j(\varphi\d^N\varphi^{-1})=\varphi(\gamma\d^N+\Delta_N-\d^N\gamma)\varphi^{-1}=0, $$
which means that $(N,\d'^N)$ equals $B \otimes _AM$  with $M$ having DG $A$-module structure defined by  $\d ^M = \d'^N |_M$. 
See Lemma \ref{eq-lem}. 
Hence $(N,\d^N) \cong B \otimes _A M$ as DG $B$-modules. 
This proves that $N$ is liftable to $A$.

\end{proof}

In the rest of the paper we consider the uniqueness of liftings.
The following lemma will be necessary for this purpose. 

\begin{lem}\label{cycle}
Let $M$ and $M'$ be DG $A$-modules, and let $\varphi: B\otimes_A M \to B\otimes_A M'$ be a graded $B$-linear homomorphism of degree $0$.
Assume we have an expansion $\varphi|_M=\sum_{i \ge0} X^{(i)}\varphi_i$, where each $\varphi _i : M \to M$ is a graded $A$-linear homomorphism. 
Then the following statements hold $:$
\begin{enumerate}
\item[$(1)$] 
$\varphi$ is a cycle in $\Hom_B(B\otimes_A M,B\otimes_A M')$ if and only if the following equalities hold for $i \ge 0$$:$ 
$$\varphi_i\d^{M}=t\varphi_{i+1}+\d^{M'}\varphi_i.$$ 
\item[$(2)$]  
$\varphi$ is a boundary in $\Hom_B(B\otimes_A M,B\otimes_A M')$ if and only if there is a graded $B$-linear homomorphism $\gamma$ of degree $1$ 
such that 
$\gamma$ has an expansion  $\gamma|_M=\sum_{i \ge0}X^{(i)}\gamma_i$, and there are equalities for $i \ge 0$$:$ 
$$
\varphi_{i}=\gamma_i\d^M+\d^{M'}\gamma_i + t\gamma_{i+1}. 
$$ 
\end{enumerate}
\end{lem}

\begin{proof}
A direct computation implies that 
\begin{equation}\label{a}
\varphi\tilde{\d}^M|_M=\varphi\d^M=\sum_{i \ge0}X^{(i)}\varphi_{i}\d^M, 
\end{equation}
where  $\tilde{\d^M}$ is the extended derivation of $\d^M$ to $B \otimes _A M$ by means of Lemma \ref{extension}. 
On the other hand,
\begin{equation}\label{b}
\tilde{\d}^{M'}\varphi|_M  =\tilde{\d}^{M'}\left( \sum_{i \ge0}X^{(i)}\varphi_i\right)
=\sum_{i \ge 0}X^{(i)}\left\{t\varphi_{i+1}+\d^{M'}\varphi_{i} \right\}.
\end{equation}

$(1)$
Since $\varphi\tilde{\d}^M-\tilde{\d}^{M'}\varphi$ is in $\mathcal{E}$, 
the cycle condition $\varphi\tilde{\d}^M-\tilde{\d}^{M'}\varphi=0$ is equivalent to that $\varphi\tilde{\d}^M|_M-\tilde{\d}^{M'}\varphi|_M=0$ by Lemma \ref
{AB-linear}.
Therefore the right hand sides of (\ref{a}) and (\ref{b}) are equal. 

$(2)$
$\varphi$ is a boundary if and only if there exists a graded $B$-linear homomorphism $\gamma$ of degree $1$ such that $\varphi=\gamma\tilde{\d}^M+\tilde{\d}^{M'}\gamma$.
Because $\varphi$ and $\gamma\tilde{\d}^M+\tilde{\d}^{M'}\gamma$ are belonging to $\mathcal{E}$, this is equivalent to that $\varphi|_M=(\gamma\tilde{\d}^M + \tilde{\d}^{M'}\gamma)|_M$ by Lemma \ref{AB-linear}.
Then the same argument as in $(1)$ using (\ref{a}) and (\ref{b}), we can show the desired equalities.
\end{proof}

\begin{thm}\label{unique}
Let $N$ be a semi-free DG $B$-module as before, and assume that  $N$ is liftable to $A$. 
If $\Ext^{|X|}_B(N,N)=0$, then a lifting of $N$ is unique up to DG isomorphisms over $A$.
\end{thm}

\begin{proof}
Assume that there are a couple of liftings  $(M,\d^M)$ and $(M',\d^{M'})$ of $N$. 
Then there is a DG $B$-isomorphism $\varphi:(B\otimes_A M,\tilde{\d}^{M}) \to (B\otimes_A M',\tilde{\d}^{M'})$,  where  $\tilde{\d}^{M}$, $\tilde{\d}^{M'}$ are the extended differentials of $\d^M$, $\d^{M'}$ respectively.
(See Lemma \ref{extension}.) 
We take an expansion $\varphi|_M=\sum_{i\in\Z_{\ge 0}} X^{(i)} \varphi_i$.
Since $\varphi$ is a cycle of degree $0$ in $\Hom_B(B\otimes_A M, B \otimes_A M')$, 
Lemma \ref{cycle} implies the equality  
$\varphi_n\d^M-\d^{M'}\varphi_n=t\varphi_{n+1}$ holds for each $n\ge 0$.
In particular, we have 
\begin{equation}\label{eq1}
\varphi_0\d^M-\d^{M'}\varphi_0=t\varphi_1.
\end{equation}
Since there is an equality $\varphi\tilde{\d}^M=\tilde{\d}^{M'}\varphi$, and since $j(\tilde{\d}^M)=0=j(\tilde{\d}^{M'})$, 
Proposition \ref{j-prop} leads that $j(\varphi)\tilde{\d}^{M}|_M=\tilde{\d}^{M'}j(\varphi)|_M$. 
It hence follows that 
$$j(\varphi)\tilde{\d}^{M}=\tilde{\d}^{M'}j(\varphi),$$
because 
$j(\varphi)\tilde{\d}^{M}- \tilde{\d}^{M'}j(\varphi)$ is  $B$-linear.
Thus  $j(\varphi)$ is a cycle of degree $-|X|$ in $\Hom_B(B\otimes_A M,B\otimes_A M')$, and it defines the element $[j(\varphi)]$  of  the homology module $H_{-|X|} \left( \Hom _B(B\otimes_A M,B\otimes_A M')\right)$. 
Regarding $N$ as a semi-free resolution of $B \otimes _A M$ and  $B \otimes _A M'$, we see that  
$\Hom_B(B\otimes_A M,B\otimes_A M')$ is quasi-isomorphic to $\Hom_B(N,N)$. 
Since we assume $\Ext^{|X|}_B(N,N)=0$, we have $[j(\varphi)]=0$.
Hence there is a graded $B$-linear homomorphism $\gamma:B\otimes_A M\to B\otimes_A M'$ of odd degree $-|X|+1$  such that $j(\varphi)=\gamma\tilde{\d}^M + \tilde{\d}^{M'}\gamma$.
Write $\gamma|_M=\sum_{i \ge 0} X^{(i)} \gamma_i$, and we get from Lemma \ref{cycle} that 
$\varphi_{n+1}=\gamma_n\d^M + \d^{M'}\gamma_n + t\gamma_{n+1}$ for $n\ge 0$.
In particular we have  
\begin{equation}\label{eq2}
\varphi_1=\gamma_0\d^M + \d^{M'}\gamma_0 + t\gamma_1.
\end{equation}
Note that $t^2=0$, because $|t|$ is odd. 
Then we obtain from (\ref{eq1}) and (\ref{eq2}) the equality  
$$(\varphi_0-t\gamma_0)\d^M=\d^{M'}(\varphi_0-t\gamma_0).$$
Namely $\varphi_0-t\gamma_0:(M,\d^M) \to (M',\d^{M'})$ is a  DG A-homomorphism.
Since $\varphi$ is a graded $B$-linear isomorphism,
$\varphi_0$ is an $A$-linear isomorphism as well by Lemma \ref{conv-isom}.  
Then it follows that $\varphi_0-t\gamma_0:M\to M'$ is an $A$-linear isomorphism, 
since $\varphi_0^{-1}+t\varphi_0^{-1}\gamma_0\varphi_0^{-1}$ gives its inverse. 
Therefore $\varphi_0-t\gamma_0:M\to M'$ is a DG isomorphism over $A$.
This completes the proof.
\end{proof}

\begin{ex}
Let  $R$  be a commutative ring and $x, y$ be elements of $R$.
Assume that the equalities of ideals;  $\mathrm{Ann}_R(x)=yR$ and $\mathrm{Ann}_R(y)=xR$ hold. 
Define $A$ to be  the extended DG $R$-algebra obtained from $R$ by the adjunction of the variable $Y$ of degree $1$ 
to kill the cycle  $y$, that is, 
$$A=R\langle Y|dY=y\rangle. $$
Further we denote by $B$ the extended DG algebra of $A$ by the adjunction of the variable $X$ of degree $2$ that kills the cycle $xY$, that is, 
$$B=A\langle X |dX= xY\rangle = R\langle X, Y | dY = y, dX= xY\rangle.$$
Under such circumstances, $B$ gives a DG $R$-algebra resolution of $S=R/yR$. 
Equivalently there is a DG $R$-algebra morphism $B\to S$ which is a quasi-isomorphism. 
See \cite{T}.
Now we assume that an $S$-module $N$ satisfies the condition $\Ext_S^3(N,N)=0$.
Then $N$ is regarded as a DG $B$-module through $B\to S$.
Taking a semi-free resolution $F_N \to N$ over $B$.
It is known from  \cite[(1.6)]{AFL} or  \cite[(1.3)]{AS}  that  $\Ext_S^3(N,N) \cong \Ext_B^3(F_N,F_N)$. 
Therefore our main theorem (Theorem \ref{main}) forces the existence of a semi-free DG $A$-module $M$ with the property       $F_N\cong B\otimes_A M$ as DG $B$-modules. 
\end{ex}

\end{document}